\documentclass[12pt]{amsart}
\usepackage{amssymb,amsmath,amsfonts,latexsym}
\usepackage{bm, enumerate}

\newcommand{\newsection}[1]{\setcounter{equation}{0} \section{#1}}
\newcommand{\clb}{\mathcal{B}}

\newcommand{\cle}{\mathcal{E}}

\newcommand{\clh}{\mathcal{H}}

\newcommand{\cln}{\mathcal{N}}

\newcommand{\clw}{\mathcal{W}}

\newcommand{\D}{\mathbb{D}}

\newcommand{\C}{\mathbb{C}}
\newcommand{\bS}{\mathbb{S}}

\newcommand{\raro}{\rightarrow}

\newcommand{\NI}{\noindent}

\newcommand\la{{\langle }}
\newcommand\ra{{\rangle}}

\newtheorem{Theorem}{\sc Theorem}[section]
\newtheorem{Lemma}[Theorem]{\sc Lemma}
\newtheorem{Proposition}[Theorem]{\sc Proposition}
\newtheorem{Corollary}[Theorem]{\sc Corollary}
\newtheorem{Example}[Theorem]{\sc Example}
\newtheorem{Remark}[Theorem]{\sc Remark}

\newtheorem{Note}[Theorem]{\sc Note}
\newtheorem{Question}{\sc Question}
\newtheorem{ass}[Theorem]{\sc Assumption}
\newtheorem{Definition}[Theorem]{\sc Definition}
\newcommand{\bt}{\begin{Theorem}}
\def\beginlem{\begin{Lemma}}
\def\beginprop{\begin{Proposition}}
\def\begincor{\begin{Corollary}}
\def\begindef{\begin{Definition}}
\def\beginexamp{\begin{Example}}
\def\beginrem{\begin{Remark}}
\def\beginq{\begin{Question}}
\def\beginass{\begin{ass}}
\def\beginnote{\begin{Note}}
\newcommand{\et}{\end{Theorem}}
\def\endlem{\end{Lemma}}
\def\endprop{\end{Proposition}}
\def\endcor{\end{Corollary}}
\def\enddef{\end{Definition}}
\def\endexamp{\end{Example}}
\def\endrem{\end{Remark}}
\def\endq{\end{Question}}
\def\endass{\end{ass}}
\def\endnote{\end{Note}}

\begin{document}

\title{Left-invertibility of rank-one perturbations}


\author[Das]{Susmita Das}
\address{Indian Statistical Institute, Statistics and Mathematics Unit, 8th Mile, Mysore Road, Bangalore, 560059,
India}
\email{susmita.das.puremath@gmail.com}

\author[Sarkar]{Jaydeb Sarkar}
\address{Indian Statistical Institute, Statistics and Mathematics Unit, 8th Mile, Mysore Road, Bangalore, 560059,
	India}
\email{jay@isibang.ac.in, jaydeb@gmail.com}

	

	
\begin{abstract}
For each isometry $V$ acting on some Hilbert space and a pair of vectors $f$ and $g$ in the same Hilbert space, we associate a nonnegative number $c(V;f,g)$ defined by
\[
c(V; f,g) = (\|f\|^2 - \|V^*f\|^2) \|g\|^2 + |1 + \langle V^*f , g\rangle|^2.
\]
We prove that the rank-one perturbation $V + f \otimes g$ is left-invertible if and only if
\[
c(V;f,g) \neq 0.
\]
We also consider examples of rank-one perturbations of isometries that are shift on some Hilbert space of analytic functions. Here, shift refers to the operator of multiplication by the coordinate function $z$. Finally, we examine $D + f \otimes g$, where $D$ is a diagonal operator with nonzero diagonal entries and $f$ and $g$ are vectors with nonzero Fourier coefficients. We prove that $D + f\otimes g$ is left-invertible if and only if $D+f\otimes g$ is invertible.
\end{abstract}
	
\maketitle

\tableofcontents

\newsection{Introduction }\label{sec: intro}

Rank-one operators are the simplest as well as easy to spot among all bounded linear operators on Hilbert spaces. Indeed, for each pair of nonzero vectors $f$ and $g$ in a Hilbert space $\clh$, one can associate a \textit{rank-one} operator $f \otimes g \in \clb(\clh)$ defined by
\[
(f \otimes g) h = \langle h, g \rangle f \qquad (h \in \clh).
\]
These are the only operators whose range spaces are one-dimensional. Here $\clb(\clh)$ denotes the algebra of all bounded linear operators on $\clh$. All Hilbert spaces in this paper are assumed to be infinite dimensional, separable, and over $\C$. Note that finite-rank operators, that is, linear sums of rank-one operators are norm dense in the ideal of compact operators, where one of the most important and natural examples of a noncompact operator is an isometry: A linear operator $V$ on $\clh$ is an \textit{isometry} if $\|Vh\|=\|h\|$ for all $h \in \clh$, or equivalently
\[
V^*V = I_{\clh}.
\]
Along this line, left-invertible operators (also known as, by a slight abuse of terminology, ``operators close to an isometry'' \cite{Shimorin}) are also natural examples of noncompact operators: $T \in \clb(\clh)$ is \textit{left-invertible} if $T$ is bounded below, that is, there exists $\epsilon >0$ such that $\|Th\| \geq \epsilon \|h\|$ for all $h \in \clh$, or equivalently, there exists $S \in \clb(\clh)$ such that
\[
ST = I_{\clh}.
\]
The intent of this paper is to make a modest contribution to the delicate structure of rank-one perturbations of bounded linear operators \cite{T Kato}. More specifically, this paper aims to introduce some methods for the left-invertibility of rank-one perturbations of isometries and, to some extent, diagonal operators. The following is the central question that interests us:

\begin{Question}
Find necessary and sufficient conditions for left-invertibility of the rank-one perturbation $V + f \otimes g$, where $V \in \clb(\clh)$ is an isometry or a diagonal operator and $f$ and $g$ are vectors in $\clh$.
\end{Question}

The answer to this question is completely known for isometries. Given an isometry $V \in \clb(\clh)$ and vectors $f, g \in \clh$, the perturbation $X = V + f \otimes g$ is an isometry if and only if there exist a unit vector $h \in \clh$ and a scalar $\alpha$ of modulus one such that $f = (\alpha - 1) h$ and $g = V^* h$. In other words, a rank-one perturbation $X$ of the isometry $V$ is an isometry if and only if there exists a unit vector $f \in \clh$ and a scalar $\alpha$ of modulus one such that
\begin{equation}\label{eqn: Iso Old classif}
X = V + (\alpha - 1) f \otimes V^*f.
\end{equation}
This result is due to Nakamura \cite{Nakamura, Nakamura 1} (and also see \cite{Serban Turcu}). For more on rank-one perturbations of isometries and related studies, we refer the reader to \cite{BT1, Choi, Clark, Fuhrmann} and also \cite{Ionascu}.

In this paper, we extend the above idea to a more general setting of left-invertibility of rank-one perturbations of isometries. In this case, however, left-invertibility of rank-one perturbations of isometries completely relies on certain real numbers. More specifically, given an isometry $V \in \clb(\clh)$ and a pair of vectors $f$ and $g$ in $\clh$, we associate a real number $c(V; f,g)$ defined by
\begin{equation}\label{eqn: def of c}
c(V; f,g) = (\|f\|^2 - \|V^*f\|^2) \|g\|^2 + |1 + \la V^*f , g\ra|^2.
\end{equation}
This is the number which precisely determine the left-invertibility of $V + f \otimes  g$:

\begin{Theorem}\label{thm: Left inv}	
Let	$V \in \clb(\clh)$ be an isometry, and let $f$ and $g$ be vectors in $\clh$. Then $V + f \otimes g$ is left invertible if and only if
\[
c(V; f,g) \neq 0.
\]
\end{Theorem}

Note that since $V$ is an isometry, we have $\|V^* f\| \leq \|f\|$, and hence, the quantity $c(V;f,g)$ is always nonnegative. Therefore, the condition $c(V; f,g) \neq 0$ in the above theorem can be rephrased as saying that $c(V; f,g) > 0$, or equivalently, $\|V^* f\| < \|f\|$ or $1 + \la V^*f , g\ra \neq 0$. However, in what follows, we will keep the constant $c(V; f,g)$ in our consideration. Not only $c(V; f,g)$ plays a direct role in the proof of the above theorem but, as we will see in Remark \ref{rem: Def of L}, this quantity also appears in the explicit representation of a left inverse of a left-invertible perturbation.

The following conclusion is now easy:

\begin{Corollary}
Let	$V \in \clb(\clh)$ be an isometry, and let $f$ and $g$ be vectors in $\clh$. Then $V + f \otimes g$ is not left-invertible if and only if
\[
\|V^* f\| = \|f\| \text{ and } \la V^*f , g\ra = -1.
\]
\end{Corollary}
	
The above theorem also provides us with a rich source of natural examples of left-invertible operators. For instance, let us denote by $\D$ the open unit disc in $\C$. Consider the shift $M_z$ on the $\cle$-valued Hardy space $H^2_{\cle}(\D)$ over $\D$, where $\cle$ is a Hilbert space. Then for any
\[
\eta \in \ker M_z^* = \cle \subseteq H^2_{\cle}(\D),
\]
and nonzero vector $g \in H^2_{\cle}(\D)$, the rank-one perturbation $M_z + \eta \otimes g$ is left-invertible. A similar conclusion holds if $f, g \in H^2(\D)$ and
\[
\la M_z^*f , g\ra \neq -1.
\]
Section \ref{sec: LI} contains the proof of the above theorem. In Section \ref{sec: Analytic}, we discuss a follow-up question: Characterizations of shifts that are rank-one perturbations of isometries. Here a shift refers to the multiplication operator $M_z$ on some Hilbert space of analytic functions (that is, a reproducing kernel Hilbert space) on a domain in $\C$. Note, however, that our analysis will be mostly limited to the level of elementary examples.

In Section \ref{sec: diaonal}, we study rank-one perturbations of diagonal operators. It is well known that the structure of rank-one perturbations of diagonal operators is also complicated (cf. \cite{Albrecht, Fang and Xia, Ionascu}). Moreover, comparison between perturbations of diagonal operators and that of isometries is perhaps inevitable if one views diagonals as normal operators and isometries as one of the best tractable non-normal operators. Here we consider $D + f \otimes g$ on some Hilbert space $\clh$, where $D$ is a diagonal operator with nonzero diagonal entries with respect to an orthonormal basis $\{e_n\}_{n=0}^\infty$ of $\clh$. We also assume that the Fourier coefficients of $f$ and $g$ with respect to $\{e_n\}_{n=0}^\infty$ are nonzero. In Theorem \ref{thm: inv}, we prove:

\begin{Theorem}
$D + f\otimes g$ is left-invertible if and only if $D+f\otimes g$ is invertible.
\end{Theorem}

In Section \ref{sec: Davidson}, we observe that the parameterized spaces considered in the work of Davidson, Paulsen, Raghupathi and Singh \cite{Paulsen Raghupati} is connected to rank-one perturbations of isometries. In the final section, Section \ref{sec: concluding}, we compute $c(V;f,g)$ when $V + f \otimes g$ is an isometry and make some further comments on rank-one perturbations of diagonal operators.

Finally, we remark that the last two decades have witnessed more intense interest in the theory of left-invertible operators starting from the work of Shimorin \cite{Shimorin}. For instance, see \cite{Badea} and references therein.

\newsection{Proof of Theorem \ref{thm: Left inv}}\label{sec: LI}

In this section, we present the proof of the left-invertibility criterion of rank-one perturbations of isometries. First note that by expanding the right-hand side of \eqref{eqn: def of c}, we have
\begin{equation}\label{eqn: cVfg}
c(V; f,g) = 1 + \|f\|^2\|g\|^2 + 2 \text{Re} \la V^*f, g\ra + |\la V^*f , g\ra|^2 - \|V^*f\|^2 \|g\|^2.
\end{equation}
Next, we make a list of the most commonly used rank-one operator arithmetic, which will be used several times in what follows. Let $f, g\in \clh$ and let $T \in \clb(\clh)$. The following holds true:
\begin{enumerate}
\item $(f \otimes g)^* = g \otimes f$.
\item ${\alpha} (f \otimes g) = ({\alpha} f) \otimes g = f \otimes (\bar{\alpha} g)$ for all $\alpha \in \C$.
\item $(f \otimes g) (f_1 \otimes g_1) = \langle f_1, g \rangle f \otimes g_1$ for all $f_1, g_1 \in \clh$.
\item $T (f \otimes g) = (Tf) \otimes g$ and so $(f \otimes g) T = f \otimes (T^*g)$.
\item $\|f \otimes g\| = \|f\| \|g\|$.
\end{enumerate}

Of course, part (2) is a particular case of part (4).

Finally, we note that $T \in \clb(\clh)$ is left-invertible if and only if $T^*T$ is invertible. Indeed, if $T$ is left-invertible, then $T^*T$ is an injective positive operator. Since $T$ is bounded below, we know that $T^*T$ is also bounded below and hence of closed range. Therefore, $T^*T$ is invertible. Conversely, suppose $X$ is the inverse of $T^*T$. Then $(XT^*) T = I$ implies that $T$ is left-invertible.

We are now ready for the proof of the theorem.

\begin{proof}[Proof of Theorem \ref{thm: Left inv}] The statement trivially holds for $f = 0$ or $g = 0$. So assume that both $f$ and $g$ are nonzero vectors. Suppose that $V + f \otimes g$ on $\clh$ is left-invertible. Then $(V+ f \otimes g)^*(V+ f \otimes g)$ is invertible with the inverse, say $L$. We have
\[
I = L (V+ f \otimes g)^*(V + f \otimes g) = L (V ^* + g \otimes f)(V+ f \otimes g).
\]
Since $V^* V = I$, it follows that
\[
\begin{split}
I & = L (V ^* + g \otimes f)(V+ f \otimes g)
\\
& = L(I +  V^*f \otimes g + g \otimes V^*f + \|f\|^2 g \otimes g)
\\
& = L + (LV^*f) \otimes g + Lg \otimes V^*f +\|f\|^2 Lg \otimes g.
\end{split}
\]
In particular, evaluating both sides on the vector $V^* f$ and $g$, respectively, we get
\[
\begin{split}
V^*f & = L V^*f + \la V^*f, g\ra L V^*f + \|V^*f\|^2 Lg+\|f\|^2\la V^*f ,g\ra Lg
\\
& = (\la V^*f, g \ra + 1)L V^*f +(\|V^*f\|^2+\|f\|^2 \la V^*f ,g \ra) Lg,
\end{split}
\]
and
\[
\begin{split}
g & = Lg + \|g\|^2 L V^*f + \la g,V^*f\ra Lg +\|f\|^2\|g\|^2 Lg
\\
& = \|g\|^2 L V^*f + (1 + \la g, V^*f\ra + \|f\|^2\|g\|^2)Lg
\\
& = \|g\|^2 L V^*f + \alpha Lg,
\end{split}
\]
where $\alpha = 1 + \la g, V^*f\ra + \|f\|^2\|g\|^2$. The latter equality implies that
\[
L V^*f = \frac{1}{\|g\|^2} (I - \alpha L)g.
\]
Now plug the value for $L V^*f$ into the expression for $V^*f$ above to get
\[
V^*f = \frac{1}{\|g\|^2} (1 + \la V^*f, g \ra) (I - \alpha L)g + ( \|V^*f\|^2+\|f\|^2 \la V^*f ,g \ra)Lg
\]
A little rearrangement then shows that
\begin{equation}\label{eqn: MP 1}
V^*f = \frac{1}{\|g\|^2} \Big(1 + \la V^*f, g \ra\Big) g + \Big( \|V^*f\|^2 + \|f\|^2 \la V^*f ,g \ra - \frac{\alpha}{\|g\|^2} (1 + \la V^*f, g \ra) \Big)Lg.
\end{equation}
We compute
\[
\begin{split}
\alpha (1 + \la V^*f, g \ra) & =  (1 + \la V^*f, g \ra) (1 + \la g, V^*f\ra + \|f\|^2\|g\|^2)
\\
& = \la V^*f, g \ra \|f\|^2\|g\|^2 + 2 \text{Re}\la V^*f, g \ra + |\la V^*f, g \ra|^2 + \|f\|^2\|g\|^2 + 1
\\
& = \la V^*f, g \ra \|f\|^2\|g\|^2 + \|V^*f\|^2 \|g\|^2 + c(V; f,g),
\end{split}
\]
where the last equality follows from the definition of $c(V; f,g)$ as in \eqref{eqn: cVfg}. Now we simplify the coefficient of $Lg$, say $a$, in the right-hand side of \eqref{eqn: MP 1} as follows:
\[
\begin{split}
a & = \|V^*f\|^2 + \|f\|^2 \la V^*f, g\ra - \frac{1}{\|g\|^2} \Big(\la V^*f, g \ra \|f\|^2\|g\|^2 + \|V^*f\|^2 \|g\|^2 + c(V; f,g) \Big)
\\
& = - \frac{1}{\|g\|^2}  c(V; f,g).
\end{split}
\]
Consequently, by \eqref{eqn: MP 1}, we have
\[
V^*f = \frac{1}{\|g\|^2} (1 + \la V^*f, g \ra) g - c(V; f,g) \frac{1}{\|g\|^2} Lg.
\]
Suppose if possible that $c(V; f,g) = 0$. Then $V^*f = \frac{1}{\|g\|^2} (1 + \la V^*f, g \ra) g$, and so
\[
\begin{split}
\langle V^*f, g \rangle & = \frac{1}{\|g\|^2} \langle (1 + \la V^*f, g \ra) g, g \rangle
\\
& = 1 + \la V^*f, g \ra,
\end{split}
\]
which is absurd. This contradiction proves that $c(V; f,g) \neq 0$.

\NI Conversely, suppose that $c:=c(V; f,g) \neq 0$. Set $R = (1+ \la g, V^*f\ra)V^*f \otimes g$, and let
\begin{equation}\label{eqn: def X}
X = I + \frac{1}{c} \{\|g\|^2 V^*f \otimes V^*f +(\|V^*f\|^2-\|f\|^2)g \otimes g - (R + R^*)\}.
\end{equation}
We claim that $X (V + f\otimes g)^*$ is a left inverse of $V + f\otimes g$, that is
\[
X (V + f\otimes g)^* (V + f\otimes g) = I.
\]
Indeed, the left hand side of the above simplifies to
\[
\begin{split}
X (V + f\otimes g)^* (V + f\otimes g) & = X (V^* + g \otimes f) (V + f\otimes g)
\\
& = X (I + g \otimes V^*f + V^*f \otimes g + \|f\|^2 g \otimes g)
\\
& = \Big(I + \frac{1}{c} \{\|g\|^2 V^*f \otimes V^*f +(\|V^*f\|^2-\|f\|^2)g \otimes g
\\
& \qquad - (R + R^*)\}\Big) \Big(I + g \otimes V^*f + V^*f \otimes g + \|f\|^2 g \otimes g\Big),
\end{split}
\]
and hence, there exists scalars $a_1, a_2, a_3$, and $a_4$ such that
\[
X (V + f\otimes g)^* (V + f\otimes g) = I + a_1 g \otimes g + a_2 V^*f \otimes g + a_3 g \otimes V^*f + a_4 V^*f \otimes V^*f.
\]
It is now enough to show that $a_1 = a_2 = a_3 = a_4 = 0$. Before getting to the proof of this claim, let us observe that
\[
R + R^* = (1+ \bar{\beta})V^*f \otimes g + (1+ \beta) g \otimes  V^*f,
\]
where $\beta := \la V^*f, g \ra$. Now we prove that $a_1 = 0$:
\[
\begin{split}
a_1 & = \text{coefficient of } g \otimes g
\\
& = \|f\|^2 + \frac{1}{c} \Big\{ - (1+ \beta) (\|V^*f\|^2 + \bar{\beta} \|f\|^2) + (\| V^*f\|^2- \|f\|^2) \Big((1+ \beta) + \|f\|^2\|g\|^2\Big)\Big\}
\\
& = \|f\|^2 + \frac{1}{c} \Big\{ - \bar{\beta} (1+ \beta) \|f\|^2 + \|V^*f\|^2 \|f\|^2\|g\|^2 - (1+ \beta) \|f\|^2 - \|f\|^4 \|g\|^2\Big\}
\\
& = \|f\|^2 + \frac{\|f\|^2}{c} \Big\{- \bar{\beta} (1+ \beta) + \|V^*f\|^2 \|g\|^2 - (1+ \beta) - \|f\|^2 \|g\|^2\Big\}
\\
& = \|f\|^2 + \frac{\|f\|^2}{c} (-c)
\\
& =0,
\end{split}
\]
where the last but one equality follows from \eqref{eqn: cVfg}. Next we compute $a_2$:
\[
\begin{split}
a_2 & = \text{coefficient of } V^*f \otimes g
\\
& = 1 + \frac{1}{c} \Big\{ \|g\|^2 (\|V^*f\|^2 + \bar{\beta} \|f\|^2) - (1+ \bar{\beta}) \Big((1 + \beta) + \|f\|^2\|g\|^2\Big)\Big\}
\\
& = 1 + \frac{1}{c} \Big\{\|V^*f\|^2 \|g\|^2 - |1 + \beta|^2 - \|f\|^2\|g\|^2 \Big\}
\\
& =0,
\end{split}
\]
as $\beta = \la V^*f, g \ra$. We turn now to compute $a_3$:
\[
\begin{split}
a_3 & = \text{coefficient of } g \otimes V^*f
\\
& = 1 + \frac{1}{c} \Big\{- (1 + \beta) \bar{\beta} - (1 + \beta) + (\| V^*f\|^2- \|f\|^2) \|g\|^2 \Big\}
\\
& = 1 + \frac{1}{c} \Big\{- |1 + \beta|^2 + (\| V^*f\|^2- \|f\|^2) \|g\|^2 \Big\}
\\
& =0,
\end{split}
\]
and, finally
\[
\begin{split}
a_4 & = \text{coefficient of } V^*f \otimes V^*f
\\
& = - \frac{1}{c} \Big\{\|g\|^2 (1 + \bar{\beta}) - (1 + \bar{\beta}) \|g\|^2 \Big\}
\\
& =0.
\end{split}
\]
This completes the proof of the fact that $V + f\otimes g$ is left-invertible with $X (V + f\otimes g)^*$ as a left inverse.
\end{proof}

\begin{Remark}\label{rem: Def of L}
From the definition of $X$ in \eqref{eqn: def X}, it is clear that if $V + f \otimes g$ is left-invertible for some isometry $V \in \clb(\clh)$ and vectors $f$ and $g$ in $\clh$, then
\[
L = \Big( I + \frac{1}{c} \{\|g\|^2 V^*f \otimes V^*f +(\|V^*f\|^2-\|f\|^2)g \otimes g - (R + R^*)\}\Big) \Big(V + f \otimes g\Big)^*,
\]
is a left-inverse of $V + f \otimes g$, where $c = c(V;f,g)$ and $R = (1 + \langle g, V^*f \rangle) V^*f \otimes g$.
\end{Remark}

It is worthwhile to observe that for an isometry $V \in \clb(\clh)$ and a vector $f \in \clh$, we have $\|V^* f\| = \|f\|$ if and only if $f \in \text{ran}V$. In particular, Theorem \ref{thm: Left inv} yields the following:

\begin{Corollary}
Let $V \in \clb(\clh)$ be an isometry and let $f$ and $g$ are nonzero vectors in $\clh$. If $f \notin \text{ran} V$, then $V + f \otimes g$ is left-invertible.
\end{Corollary}

\newsection{Analytic operators}\label{sec: Analytic}

Recall that an isometry $V \in \clb(\clh)$ is called a \textit{pure isometry} if
\[
\bigcap_{n = 0}^\infty V^n \clh = \{0\}.
\]
As we will see soon, this is also known as the \textit{analytic property} of $V$. It is known that an isometry $V \in \clb(\clh)$ is pure if and only if $V$ is unitarily equivalent to $M_z$ on the $\clw$-valued Hardy space $H^2_{\clw}(\D)$, where $\clw = \ker V^*$ is the \textit{wandering subspace} corresponding to $V$. Here $M_z$ denotes the multiplication operator by the coordinate function $z$ on $H^2_{\clw}(\D)$ (see \eqref{eqn: Mz} below). Rank-one perturbations of isometries (or pure isometries) that are pure isometries form a rich class of operators and are fairly complex in nature \cite{Nakamura 1}. The methods involve heavy machinery of $H^\infty(\D)$-function theory, which is mostly unavailable for general function spaces (see \cite{BT1, Choi, Fuhrmann, Ionascu, Serban Turcu}). In this section we discuss some examples of rank-one perturbations of isometries that are shift or simply analytic.

We begin with a brief introduction to shift operators on reproducing kernel Hilbert spaces. Let $\cle$ be a Hilbert space and $\Omega$ be a domain in $\C$. Let $\clh$ be a Hilbert space of $\cle$-valued analytic functions on $\Omega$. Suppose the \textit{evaluation map}
\[
ev_w (f) = f(w) \quad \quad (f \in \clh),
\]
defines a bounded linear operator $ev_w : \clh \raro \cle$ for all $w \in \Omega$. Then the \textit{kernel} function $k : \Omega \times \Omega \raro \clb(\cle)$ defined by
\[
k(z,w) = ev_z \circ ev_w^* \qquad (z, w \in \Omega),
\]
is \textit{positive definite}, that is,
\[
\sum_{i,j=1}^n \langle k(z_i, z_j) \eta_j, \eta_i \rangle_{\cle} \geq 0,
\]
for all $\{z_i\}_{i=1}^n \subseteq \Omega$, $\{\eta_i\}_{i=1}^n \subseteq \cle$ and $n \geq 1$. Moreover, $k$ is analytic in the first variable and satisfies the \textit{reproducing property}
\[
\langle ev_w(f), \eta \rangle_{\cle} = \langle f(w), \eta \rangle_{\cle} = \langle f, k(\cdot, w)\eta\rangle_{\clh},
\]
for all $f \in \clh$, $w \in \Omega$ and $\eta \in \cle$. We denote the space $\clh$ by $\clh_k$ and call it \textit{analytic Hilbert space}. The \textit{shift operator} $M_z$ on $\clh_k$ is defined by
\begin{equation}\label{eqn: Mz}
(M_z f)(w) = w f(w) \quad \quad (f \in \clh_k, w \in \Omega).
\end{equation}
We always assume that $M_z$ is a bounded linear operator on $\clh_k$ (equivalently, $z \clh_k \subseteq \clh_k$). It is easy to see that if $M_z$ is a shift on some $\clh_k$, then
\[
\bigcap_{n = 0}^\infty M_z^n \clh_k = \bigcap_{n = 0}^\infty z^n \clh_k = \{0\}.
\]
This is the property which bridges the gap between left-invertible operators and left-invertible shifts. More precisely, following the ideas of Shimorin \cite{Shimorin}, a bounded linear operator $T$ on $\clh$ is called \textit{analytic} if
\[
\bigcap_{n = 0}^\infty T^n \clh = \{0\}.
\]
If $T \in \clb(\clh)$ is a left-invertible analytic operator, then there exists an analytic Hilbert space $\clh_k$ such that $T$ and the shift $M_z$ on $\clh_k$ are unitarily equivalent \cite{Shimorin}. Therefore, up to unitary equivalence, analytic left-invertible operators are nothing but left-invertible shifts.

The following proposition collects some examples of analytic and shift operators.

\begin{Proposition}\label{Aly 2}
Let $V \in \clb(\clh)$ be a pure isometry, $m, n \in \mathbb{Z}_+$, and let $f_0 \in \ker V^*$. If $S = V+ V^m f_0 \otimes V^n f_0$, then the following holds:
\begin{enumerate}
\item $S$ is analytic whenever $m > n$.
\item $S$ is a shift whenever $m > n+1$.
\end{enumerate}
\end{Proposition}
\begin{proof}
For simplicity, for each $t \in \mathbb{Z}$, we set
\[
f_t = \begin{cases}
V^t f_0  & \mbox{if } t \geq 0
\\
V^{*-t} f_0 & \mbox{if } t < 0.
\end{cases}
\]
Since $f_0 \in \ker V^*$, it follows that $f_t = 0$ for all $t <0$. Suppose $m >n$. Observe that $\la f_m , f_n\ra  = \la V^m f_0 , V^n f_0\ra = 0$, and hence
\[
\begin{split}
S^2 & = V^2 + f_{m+1} \otimes f_n + f_m \otimes f_{n-1} + \la f_m , f_n\ra f_m \otimes f_n
\\
&=V^2 + f_m \otimes f_{n-1} + f_{m+1} \otimes f_n.
\end{split}
\]
Then, by induction, we have
\[
S^{k+1} = V^{k+1}+ f_m \otimes f_{n-k} + f_{m+1} \otimes f_{n-k+1} + \cdots + f_{m+k-1} \otimes f_{n-1} + f_{m+k} \otimes f_n,
\]
that is
\begin{equation}\label{eqn: 3.2}
S^{k+1} = V^{k+1} + \sum_{j=0}^{k} f_{m+j} \otimes f_{n-k+j},
\end{equation}
for all $k \geq 1$. In particular, if $k=n+j$ and $j \geq 1$, then it follows that
\[
S^{n+j+1} = V^{n+j+1} + f_{m} \otimes f_{-j} + f_{m+1} \otimes f_{-j+1} + \cdots + f_{m+ n+ j-1} \otimes f_{n-1} + f_{m+n+j} \otimes f_n.
\]
At this point, we note that $f_{-p} = 0$ for all $p >0$, and hence
\[
\begin{split}
S^{n+j+1} & = V^{n+j+1} + f_{m+j} \otimes f_0 + f_{m+j+1} \otimes f_1 + \cdots + f_{m+j+n-1} \otimes f_{n-1} + f_{m+n+j} \otimes f_n
\\
& = V^{n+j+1}(I + \sum_{i=0}^{n} f_{m-n-1+i} \otimes f_i),
\end{split}
\]
as $m >n$. This implies that
\[
S^{n+j+1}\clh \subseteq V^{n+j+1} \clh \qquad (j \geq 1).
\]
From here we see that
\[
\bigcap_{r \geq 0} S^r \clh \subseteq
\bigcap_{r \geq {n+1}} S^r \clh \subseteq
\bigcap_{r \geq {n+1}} V^r\clh = \{0\},
\]
where the last equality follows from the fact that $V$ is pure. To prove (2), we compute the value of $c(V; f,g)$ with $f = V^m f_0$ and $g = V^n f_0$:
\[
\begin{split}
c(V; f,g) & = (\|f\|^2 - \|V^*f\|^2) \|g\|^2 + |1 + \la V^*f , g\ra|^2
\\
& = (\|V^m f_0\|^2 - \|V^{m-1} f_0\|^2) \|V^n f_0\|^2 + |1 + \la V^{m-1}f_0 , V^n f_0\ra|^2
\\
& = 0 \times \|f_0\|^2 + |1 + 0|
\\
& = 1,
\end{split}
\]
where the last but one equality follows because $m-n-1  >0$ implies $\langle V^{*n} V^{m-1} f_0, f_0 \rangle =0$. The first part and Theorem \ref{thm: Left inv} then completes the proof of part (2). 	
\end{proof}

The above observation is fairly elementary. The general classification of rank-one perturbations of isometries (or pure isometries) that are shift on some reproducing kernel Hilbert space is an open problem. However, see \cite[Theorem 1]{Nakamura} and \cite{Nakamura 1} in the context of classifications of rank-one perturbations of isometries that are pure isometry.

The following is also a simple class of examples of analytic operators.

\begin{Proposition}\label{thm: Aly 1}
Let $V \in \clb(\clh)$ be a pure isometry, $f$ and $g$ be vectors in $\clh$, and suppose $V^*g + \la g,f\ra g =0$. Then $V + f \otimes g$ is analytic.
\end{Proposition}

\begin{proof}
If we set $S: = V + f \otimes g$, then
\[
S^2 = V^2+ Vf \otimes g + f \otimes (V^*g + \la g, f \ra g) = VS.
\]
Therefore
\[
S^{n+1} = V^n S \qquad (n \geq 1),
\]
can be proved analogously by induction. In particular
\[
S^{n+1} \clh = V^n S \clh \subseteq V^n \clh \qquad (n \geq 0),
\]
and hence, by using the fact that $V$ is a pure isometry, it follows that
\[
\bigcap_{n = 0}^\infty (V+ f \otimes g)^{n+1}\clh\subseteq \bigcap_{n = 0}^\infty V^n\clh =\{0\},
\]
that is, $V + f \otimes g$ is analytic.
\end{proof}

Note that $V^*g + \la g,f\ra g =0$  is equivalent to the condition that $g \in \ker (V + f \otimes g)^*$.

Recall that the scalar-valued Hardy space $H^2(\D)$ is a reproducing kernel Hilbert space corresponding to the Szeg\"{o} kernel $\bS: \D \times \D \raro \mathbb{C}$, where
\[
\bS(z,w) = (1 - z \bar{w})^{-1} \qquad (z, w \in \D).
\]
For each $w \in \D$, consider the analytic function $\bS(\cdot, w) : \D \raro \mathbb{C}$ defined by (also known as the kernel function, see the discussion at the beginning of this section)
\[
(\bS (\cdot, w))(z) = \bS(z,w) \qquad (z \in \D).
\]

\begin{Example}
The following examples illustrate some direct application of the above propositions.

\begin{enumerate}
\item Fix $w \in \D$, and set $g = \bS(\cdot, w)$. We know that $M_z^* \bS(\cdot, w) = \bar{w} \bS(\cdot, w)$. Choose $f \in H^2(\D)$ such that $\langle g, f \rangle_{H^2(\D)} = - \bar{w}$ (for instance, $f = \frac{-1}{\bar{w}^{n-1}} z^n$ for some $n \geq 1$). Evidently $M_z^*g + \la g,f\ra g =0$, and hence, $M_z + f \otimes \bS(\cdot, w)$ is an analytic operator.
\item Consider $f = z$ and $g = 1$ in $H^2(\D)$. Then $c(M_z; f, g) = 2 \neq 0$, and hence $M_z + f \otimes g$ is a shift.
\item Consider $f = z$ and $g = -1$ in $H^2(\D)$. Then $c(M_z; f, g) = 0$, and hence $M_z + f \otimes g$ not left-invertible, but analytic by Proposition \ref{Aly 2}.
\end{enumerate}
\end{Example}

Note that the rank-one perturbation $M_z + z^2 \otimes z$ is similar to $M_z$ on $H^2(\D)$. Here the similarity follows easily from the fact that $M_z + z^2 \otimes z$ is a weighted shift with the weight sequence $\{1, 2, 1, 1, \ldots\}$. This implies, of course, that $M_z + z^2 \otimes z$ is analytic, where on the one hand
\[
M_z^* z + \la z, z^2\ra z =1 \neq 0.
\]
Therefore, $M_z + z^2 \otimes z$ is an example of an analytic rank-one perturbation of $M_z$ which does not satisfy the hypothesis of Proposition \ref{thm: Aly 1}.

\newsection{Diagonal operators}\label{sec: diaonal}

In this section, we examine rank-one perturbations of diagonal operators. We prove that all the interesting left-invertible rank-one perturbations of diagonal operators are invertible.

Throughout this section, we fix a Hilbert space $\clh$ with orthonormal basis $\{e_n\}_{n=0}^\infty$ of $\clh$. We also fix vectors $f = \sum_{n=0}^{\infty} a_n e_n$ and $g = \sum_{n=0}^{\infty}b_n e_n$ in $\clh$ and diagonal operator $D \in \clb(\clh)$ with diagonal entries $\{\alpha_n\}_{n \geq 0}$. Also, we set
\[
T = D + f \otimes g.
\]
We will assume throughout this section that
\[
\alpha_n, a_n, b_n \neq 0 \qquad (n \geq 0),
\]
as this is the class of perturbations we all are most interested in (cf. \cite{Ionascu}). Furthermore, we define the scalar
\[
r := 1 + \sum_{n=0}^{\infty} \frac{a_n \bar{b}_n}{\alpha_n}.
\]
The following result is from Ionascu \cite[Proposition 2.4]{Ionascu}:

\begin{Proposition}\label{thm: kernel}
$T$ admits zero as an eigenvalue if and only if $r = 0$ and $\{\frac{a_n}{\alpha_n}\}_{n\geq 0}$ is a square summable sequence.
\end{Proposition}

The key to our analysis lies in the following observation which is also a result of independent interest.

\begin{Proposition}\label{thm: range}
$\{e_n\}_{n \geq 0} \subseteq \text{ran} T$ if and only if $r \neq 0$ and $\{\frac{a_n}{\alpha_n}\}_{n\geq 0}$ is a square summable sequence.
\end{Proposition}	
\begin{proof}
Assume that $e_j \in \text{ran} T$ for some arbitrary but fixed integer $j \geq 0$. Then there exists $x = \sum_{n = 0}^\infty c_n e_n \in \clh$ such that $T x = (D + f \otimes g)x = e_j$. Therefore
\begin{equation}\label{eqn: ei}
e_j = \sum_{n = 0}^\infty (c_n \alpha_n) e_n + \la x, g \ra \sum_{n = 0}^\infty a_n e_n.
\end{equation}
Note that $\la x, g \ra \neq 0$. Indeed, if $\la x, g \ra =0$, then
\[
c_n = \begin{cases}
\frac{1}{\alpha_j} & \mbox{if } n=j \\
0 & \mbox{otherwise},
\end{cases}
\]
and hence $x =\frac{1}{\alpha_j}e_j$. Since $g = \sum_{n=0}^{\infty} b_n e_n$, using $\la x, g \ra=0$, we have $b_j = 0$. This contradiction shows,
as promised, that $\la x, g \ra \neq 0$. Now equating the coefficients of terms on either side of \eqref{eqn: ei}, we have
\[
c_n = \begin{cases}
\frac{1}{\alpha_j} (1 -  a_j \la x, g \ra) & \mbox{if } n=j \\
-\frac{a_n}{\alpha_n} \la x, g \ra & \mbox{otherwise}.
\end{cases}
\]
In particular, $\{\frac{a_n}{\alpha_n}\}_{n\geq 0}$ is a square summable sequence, and, as $\la x, g \ra = \sum_{n =0}^\infty c_n \bar{b}_n$, we have
\[
\la x, g\ra = -\la x, g\ra \sum_{n = 0}^\infty \frac{a_n\bar{b}_n}{\alpha_n} + \frac{\bar{b}_j}{\alpha_j},
\]
which implies
\[
\la x, g \ra \Big(1+\sum_{n = 0}^\infty\frac{a_n\bar{b}_n}{\alpha_n}\Big) = \la x, g\ra r = \frac{\bar{b}_j}{\alpha_j},
\]
and hence $r \neq 0$.

Conversely, assume that $r \neq 0$ and $\{\frac{a_n}{\alpha_n}\}_{n\geq 0}$ is a square summable sequence. Fix an integer $j \geq 0$. Then
\[
y = - \frac{\bar{b}_j}{r\alpha_j} \Big(\sum_{n = 0}^\infty \frac{a_n}{\alpha_n} e_n \Big)+ \frac{1}{\alpha_j}e_j,
\]
is a vector in $\clh$. Note that
\[
\langle y, g \rangle = - \frac{\bar{b}_j}{r\alpha_j} (r - 1) + \frac{\bar{b}_j}{\alpha_j} = \frac{\bar{b}_j}{r\alpha_j}.
\]
Using the representation $f = \sum_{n = 0}^\infty a_n e_n$, we deduce from the above that
\[
\begin{split}
Ty = (D + f \otimes g)y = - \frac{\bar{b}_j}{r\alpha_j} \sum_{n = 0}^\infty a_ne_n +e_j +\la y, g\ra f = e_j.
\end{split}
\]
This implies that $e_j \in \text{ran}T$ for all $j \geq 0$ and completes the proof of the proposition.
\end{proof}

We also need the following lemma:
	
\begin{Lemma}\label{thm: bounded below}
If $T$ is bounded below, then $D$ is invertible.
\end{Lemma}	
	
\begin{proof}
Assume by contradiction that $\{\alpha_{n_k}\}$ is a subsequence of the sequence $\{\alpha_n\}$, which converges to zero. Now
\[
T e_{n_k} = (D+ f\otimes g)e_{n_k} = \alpha_{n_k}e_{n_k}+\la e_{n_k},g\ra f =\alpha_{n_k}e_{n_k}+b_{n_k}f,
\]
implies
\[
\|T e_{n_k}\|\leq |\alpha_{n_k}|+|b_{n_k}|\|f\|.
\]
This shows that $\{T e_{n_k}\}$ converges to zero for the sequence of unit vectors $\{e_{n_k}\}$. But this contradicts the fact that $T$ is bounded below. Therefore the sequence $\{\alpha_n\}$ has no subsequence that converges to zero. Consequently, there exists $M > 0$ such that
\[
|\alpha_n|> M \qquad (n \geq 0),
\]
and hence $\{\frac{1}{\alpha_n}\}$ is a bounded sequence. We conclude that $D$ is invertible.
\end{proof}

The converse is not true. For example, choose $ f, g \in \clh$ such that $\la f,g\ra = -1$. Then, by Proposition \ref{thm: kernel}, $I +f\otimes g$ is not injective, and hence $I +f\otimes g$ is not invertible. However, under the assumption that $D+f\otimes g$ is injective, we have the following:

\begin{Proposition}\label{Prop: bdd below}
If $D$ is bounded below and $T$ is injective, then $T$ is left-invertible.
\end{Proposition}

\begin{proof}
Assume by contradiction that $T = D+ f\otimes g$ is not bounded below. Then there is a sequence $\{h_n\}\subseteq\clh$ with $\|h_n\|=1$ such that $T h_n \raro 0$. By the compactness of $f\otimes g$, there exists a subsequence $\{h_{n_k}\}$ of $\{h_n\}$ such that $(f \otimes g) h_{n_k}$ converges. Then,  $D h_{n_k} = (T - f\otimes g) h_{n_k}$ converges. But since $D$ is bounded below, this gives us $h_{n_k} \raro \tilde{h}$ for some $\tilde h \in \clh$. In particular, we have $\|\tilde{h}\|=1$. On the other hand, since $T$ is a bounded linear operator, we have
\[
T \tilde{h} = \lim_{k \raro \infty} T h_{n_k} = 0,
\]
that is, $\tilde{h} \in \ker T$. But, $\ker T = \{0\}$ by our assumption, and hence $\tilde h = 0$, which contradicts the fact that $\|\tilde h\| = 1$. Therefore, $T$ is bounded below.	
\end{proof}

Although Proposition \ref{Prop: bdd below} is not directly related to the main result of this section, but perhaps fits appropriately with our present context. The following result and its proof are also along the same line and perhaps of independent interest.

\begin{Proposition}\label{Prop: closed range}
If $D$ has a closed range, then $T$ also has a closed range.
\end{Proposition}
\begin{proof}
Suppose $\cln = \ker T$, and suppose that $\text{ran}D$ is closed. Then $T|_{\cln^\perp}$ is injective. Assume by contradiction that $\text{ran} T$ is not closed. Then $X: = T|_{\cln^\perp}$ is not left-invertible. Proceeding exactly as in the proof of Proposition \ref{Prop: bdd below} (by replacing the role of $T$ by $X$), we will find a similar contradiction.
\end{proof}

We come now to the main result on left-invertibility of rank-one perturbations.

\begin{Theorem}\label{thm: inv}
$D + f\otimes g$ is left-invertible if and only if $D+f\otimes g$ is invertible.
\end{Theorem}

\begin{proof}
For the nontrivial direction, assume that $T = D+ f\otimes g$ is left-invertible. Assume by contradiction that $T$ is not invertible. Since, in particular, $\text{ran} T$ is closed, $\{e_n\}_{n \geq 0} \nsubseteq \text{ran} T$. Now by Proposition \ref{thm: kernel}, either $r \neq 0$ or the sequence $\{\frac{a_n}{\alpha_n}\}_{n\geq 0}$ is not square summable. On the other hand, we know from Lemma \ref{thm: bounded below} that $D$ is invertible, and hence
\[
D^{-1}f = \sum_{n=0}^{\infty} \frac{a_n}{\alpha_n} e_n \in \clh.
\]
This implies, of course, that $\{\frac{a_n}{\alpha_n}\}_{n\geq 0}$ is a square summable sequence, and hence $r \neq 0$. As a consequence, we can apply Proposition \ref{thm: range} to $T$: the basis vectors $\{e_n\}_{n \geq 0} \subseteq  \text{ran} T$; which is a contradiction. This proves that $T$ is invertible.
\end{proof}

If we know that $D$ is invertible (which anyway follows from Lemma \ref{thm: bounded below}) and $r \neq 0$, then the surjectivity of $T = D+ f\otimes g$ in the above proof also can be obtained as follows: Observe that
\[
1+ \la D^{-1}f , g \ra = 1 + \sum_{n = 0}^\infty \frac{a_n\bar{b}_n}{\alpha_n}=r.
\]
Then for each $y \in \clh$, we consider
\[
x = D^{-1}y - \frac{1}{r}\la D^{-1}y ,g\ra D^{-1}f.
\]
We deduce easily that $T x = y$, which completes the proof of the fact that $T$ is onto.

\newsection{An example}\label{sec: Davidson}

Let $T$ be a bounded linear operator on $H^2(\D)$. Suppose
\[
[T] = \begin{bmatrix}
	0 & 0 & 0  & \dots
	\\
	a_{01} & 0 & 0 & \ddots
	\\
	a_{02} & a_{12} & 0  & \ddots
	\\
	a_{03} & a_{13} & a_{23} & \ddots
	\\
	\vdots & \ddots & \ddots & \ddots
\end{bmatrix},
\]
the matrix representation of $T$ with respect to the standard orthonormal basis $\{z^n, n \geq 0\}$ of $H^2(\D)$. Clearly, $T(z^n) \subseteq z^{n+1}H^2(\D)$, and hence
\[
T^n (H^2(\D)) \subseteq z^n H^2(\D) \qquad (n\geq 0).
\]
It follows that
\[
\bigcap_{n =0}^\infty T^n H^2(\D) \subseteq \bigcap_{n =0}^\infty z^n H^2(\D) = \{0\},
\]
that is, $T$ is analytic. In particular, for each $\alpha$ and $\beta$ in $\C$, the matrix operator
\[
[T_{\alpha, \beta}] = \begin{bmatrix}
	0 & 0 & 0  & 0 & \dots
	\\
	\alpha & 0 & 0 & 0 & \ddots
	\\
	\beta & 0 & 0  & 0 & \ddots
	\\
	0 & 1 & 0 & 0 &  \ddots
	\\
	0 & 0 & 1 & 0 &  \ddots
	\\
	\vdots & \ddots & \ddots & \ddots & \ddots
\end{bmatrix},
\]
defines an analytic operator $T_{\alpha, \beta}$ on $H^2(\D)$. Moreover, one can show that
\[
T_{\alpha, \beta} = M_z^2 + (\alpha z + (\beta - 1) z^2) \otimes 1,
\]
that is, $T_{\alpha, \beta}$ is a rank-one perturbation of the shift $M_z^2$ on $H^2(\D)$. Next, we compute $c(T_{\alpha, \beta}; f, g)$, where $f = \alpha z + (\beta - 1) z^2$ and $g = 1$. Since
\[
\langle M_z^{*2} f,g \rangle_{H^2(\D)} = \beta - 1,
\]
and $\|M_z^{*2} f\|^2 = |\beta - 1|^2$, and $\|f\|^2 = |\alpha|^2 + |\beta - 1|^2$, it follows that
\[
c(T_{\alpha, \beta}, \alpha z + (\beta - 1) z^2, 1) = |\alpha|^2 + |\beta|^2.
\]
Thus we have proved:

\begin{Proposition}
Let $(\alpha, \beta) \in \C^2 \setminus \{(0,0)\}$, and suppose $f = \alpha z + (\beta - 1) z^2$ and $g = 1$. Then:

(1) $T_{\alpha, \beta}$ is a shift on $H^2(\D)$,

(2) $T_{\alpha, \beta} = M_z^2 + f \otimes g$, and

(3) $c(M_z^2; f,g) = |\alpha|^2 + |\beta|^2$.
\end{Proposition}

We recall in passing that $T_{\alpha, \beta}$ is a shift means the existence of an analytic Hilbert space $\clh_k$ and a unitary $U: H^2(\D) \raro \clh_k$ such that $T_{\alpha, \beta} = U^* M_z U$ (see the discussion preceding Proposition \ref{Aly 2}).

We continue with the matrix representation $[T_{\alpha, \beta}]$. It is immediate that $T_{\alpha, \beta}$ is an isometry if and only if
\[
|\alpha|^2+|\beta|^2 = 1.
\]
Denote by $H^2_{\alpha , \beta}(\D)$ the closed codimension one subspace of $H^2(\D)$ with orthonormal basis $\{\alpha +\beta z , z^2 ,z^3 , \ldots\}$. Clearly, $H^2_{\alpha , \beta}(\D)$ is an invariant subspace of $M_z^2$. One can verify
straightforwardly that the map $U: H^2(\D) \raro H^2_{\alpha , \beta}(\D)$ defined by
\[
U z^n = \begin{cases}
\alpha + \beta z & \mbox{if } n=0 \\
z^{n+1} & \mbox{otherwise}
\end{cases}
\]
is a unitary operator and
\[
U T_{\alpha, \beta} = M_z^2 U,
\]
that is, $T_{\alpha, \beta}$ on $H^2(\D)$ and $M_z^2|_{H^2_{\alpha , \beta}(\D)}$ on $H^2_{\alpha , \beta}(\D)$ are unitarily equivalent. The operator $M_z^2|_{H^2_{\alpha , \beta}(\D)}$ on $H^2_{\alpha , \beta}(\D)$, for $(\alpha, \beta) \in \C^2$ such that $|\alpha|^2+|\beta|^2 = 1$, has been considered in \cite{Paulsen Raghupati} in the context of invariant subspaces and a constrained Nevanlinna-Pick interpolation problem. Clearly, in the context of perturbation theory, it is worth exploring and explaining the results of \cite{Paulsen Raghupati}.

\newsection{Concluding remarks}\label{sec: concluding}

We begin by computing $c(V; f,g)$ for rank-one perturbations that are isometries. Suppose $V \in \clb(\clh)$ is an isometry and $f$ and $g$ are vectors in $\clh$. It is curious to observe that
\[
c(V; f,g) = 1,
\]
whenever $V + f \otimes g$ is an isometry. Indeed, in the present case, by \eqref{eqn: Iso Old classif}, there exist a unit vector $h \in \clh$ and a scalar $\alpha$ of modulus one such that $f = (\alpha - 1) h$ and $g = V^* h$. Then \eqref{eqn: cVfg} yields
\[
\begin{split}
c(V; f,g) -1 & = |\alpha - 1|^2 \|V^*h\|^2 + 2 (\text{Re}(\alpha -1)) \|V^* h\|^2 + |\alpha - 1|^2 \|V^*h\|^4 (1 - 1)
\\
& = (|\alpha - 1|^2 + 2 \text{Re}(\alpha -1)) \|V^*h\|^2
\\
& = 0,
\end{split}
\]
as $|\alpha| = 1$. This completes the proof of the claim.

It would be interesting to investigate the nonnegative number $c(V; f,g)$ in terms of analytic and geometric invariants, if any, of rank-one perturbations of isometries. This is perhaps a puzzling question for which we do not have any meaningful answer or guess at this moment.

We conclude this paper by making some additional comments on (non-analytic features of) perturbations of diagonal operators. The following easy-to-prove proposition says that rank-one perturbations of common diagonal operators do not fit well with shifts on reproducing kernel Hilbert spaces.

\begin{Proposition}
Let $D \in \clb(\clh)$ be a Fredholm diagonal operator, and let $f, g \in \clh$. Then $D+ f \otimes g$ cannot be represented as shift.
\end{Proposition}
\begin{proof}
Assume the contrary, that is, assume that $D + f \otimes g$ is unitarily equivalent to $M_z$ on some reproducing kernel Hilbert space $\clh_k$. Since $D$ is Fredholm, and $M_z$ and $D + f \otimes g$ are unitarily equivalent, we have ${ind}(M_z) = ind (D) = 0$. On the other hand, since $M_z$ is injective, it follows that
\[
{ind}(M_z) = \dim \ker M_z-\dim \ker M_z^* < 0,
\]
which is a contradiction.
\end{proof}

In the context of Theorem \ref{thm: inv}, we remark that rank-one perturbations of diagonal operators need not be left-invertible: Consider a compact diagonal operator $D$ (for instance, consider $D$ with diagonal entries $\{\frac{1}{n}\}$). Then a rank-one perturbation of $D$ is also compact, and hence the perturbed operator cannot be left-invertible.

In Lemma \ref{thm: bounded below}, we prove that if $D + f \otimes g$ is bounded below, then $D$ is invertible. This was one of the key tools in proving Theorem \ref{thm: inv}: $D + f \otimes g$ is left-invertible if and only if $D + f \otimes g$ is invertible. Of course, we assumed that the Fourier coefficients of $f$ and $g$ are nonzero. Here, we would like to point out that rank-one perturbation of an invertible operator need not be invertible. In fact, the invertibility property of rank-one perturbations of invertible operators can be completely classified (see \cite[Lemma 2.7]{Ionascu}): Let $D$ be an invertible diagonal operator. Then  $D + f \otimes g$ is invertible if and only if
\[
1 + \langle D^{-1} f, g \rangle \neq 0.
\]
Finally, in the context of left-invertibility, consider $D = I_{\clh}$ and choose $f$ and $g$ from $\clh$ such that $\langle f, g \rangle = -1$. It is easy to see that $c(D; f, g) = 0$, and hence, $D + f \otimes g$ is not left-invertible.

\vspace{0.1in}

\noindent\textbf{Acknowledgement:}
The research of the second named author is supported in part by Core Research Grant, File No: CRG/2019/000908, by the Science and Engineering Research Board (SERB), Department of Science \& Technology (DST), Government of India.

\vspace{0.1in}

\bibliographystyle{amsplain}

\end{document}